\documentclass[11pt]{amsart}
\usepackage{amssymb}
\usepackage{verbatim}
\usepackage{amscd}
\usepackage{amsmath}

\theoremstyle{plain}
\newtheorem{theorem}{Theorem}[section]
\newtheorem{corollary}[theorem]{Corollary}
\newtheorem{lemma}[theorem]{Lemma}
\newtheorem{proposition}[theorem]{Proposition}

\newtheorem{remark}[theorem]{Remark}

\numberwithin{equation}{section}

\def\boxit#1{\vbox{\hrule\hbox{\vrule\kern3pt
     \vbox{\kern3pt#1\kern3pt}\kern3pt\vrule}\hrule}}

\newfont{\msam}{msam10}            
\newfont{\msym}{msbm10 scaled\magstep1}            

\newfont{\gotic}{eufm10 scaled\magstep1}

\newcommand{\ra}{\rightarrow}
\newcommand{\lra}{\mbox{\Huge $\longrightarrow$}}

\newcommand{\GL}{{\rm GL}}

\newcommand{\kra}{\kern-7pt\rightarrow\kern-7pt}

\def\bfr{\mathbf{r}}

\def\bbr{{\mathbb R}}

\def\bbn{{\mathbb N}}
\def\bbc{{\mathbb C}}

\def\ra{\rightarrow}

\def\x{\times}

\def\gl{\mathrm{GL}}

\def\inv{^{-1}}

\def\lra{\longrightarrow}

\def\fm{\phantom{-}}
\numberwithin{equation}{section}

\newcommand{\angles}[1]{{\langle #1 \rangle}}
\def\begtab{\begin{tabbing} WW\=23/02: \= point 1\kill}

\def\NIL1{{\mathcal H^{3}}}

\def\wt{\widetilde}

\def\n2{\mathfrak{N}_2}

\def\frakm{{\mathfrak m}}
\def\fraku{{\mathfrak u}}
\def\frakh{{\mathfrak h}}
\def\frakg{{\mathfrak g}}

\begin{document}

\title[ Area and holonomy  on   the principal $U(n)$ bundles ]{  The topological aspect of the holonomy displacement   on   the principal $U(n)$ bundles   over Grassmannian manifolds  }

\author{Taechang Byun}
\address{Department of Mathematics and Statistics, Sejong University,
Seoul 143-747, Korea}
\email{tcbyun@gmail.com}
\author{Younggi Choi}
\address{ Department of Mathematics Education,
Seoul National University,
Seoul 151-748, Korea}
\email{yochoi@snu.ac.kr}



\maketitle

\begin{abstract}
Consider the principal $U(n)$ bundles over Grassmann manifolds
$U(n)\ra U(n+m)/U(m) \stackrel{\pi}\ra G_{n,m}$.
Given $X \in U_{m,n}(\bbc)$ and a 2-dimensional subspace 
$\frakm' \subset \frakm $ $ \subset \mathfrak{u}(m+n), $ 
assume either 
$\frakm'$ is induced by $X,Y \in U_{m,n}(\bbc)$
with $X^{*}Y = \mu I_n$ for some $\mu \in \bbr$
or
by $X,iX \in U_{m,n}(\bbc)$.  
Then $\frakm'$  gives rise to a complete totally geodesic surface $S$ in the base
space. Furthermore,
let $\gamma$
be a piecewise smooth, simple closed curve on $S$
parametrized by $0\leq t\leq 1$, and $\wt\gamma$ its horizontal lift on
the bundle
$U(n) \ra \pi^{-1}(S) \stackrel{\pi}{\rightarrow} S,$
which is immersed in
$U(n) \ra U(n+m)/U(m) \stackrel{\pi}\ra G_{n,m} $.
Then
$$
\wt\gamma(1)= \wt\gamma(0) \cdot ( e^{i \theta} I_n)
\text{\hskip24pt or\hskip12pt } \wt\gamma(1)= \wt\gamma(0),
$$
depending on whether the immersed bundle is flat or not,
where $A(\gamma)$ is the area of the region on the surface $S$
surrounded by $\gamma$ and 
$\theta= 2  \cdot \tfrac{n+m}{2n} A(\gamma).$
\end{abstract}

\section{Introduction}
For two natural numbers $n,m \in \bbn,$ let
$$
  U_{m,n}(\bbc):= 
  \{ 
     X \in M_{m,n}(\bbc) 
     \,\, | \,\, 
     X^{*} X = \lambda I_n \text{ for some } \lambda \in \bbc - \{0\}
  \},
$$
which may be regarded as a generalization of a unitary group. 
It plays an important role in studying the princiapl $U(n)$ bundles $U(n) \ra U(n+m)/U(m) \ra G_{n,m}$ over Grassmannian manifolds, 
where, for $k \in \bbn,$ $U(k)$ has a metric, related to the Killing-Cartan form, given by
$$
  \langle A,B \rangle 
   = \tfrac{1}{k} \text{Re}\big(\text{Tr}(A^{*}B)\big), 
      \qquad A,B \in \mathfrak{u}(k),
$$
and each quotient space has the induced metric which makes the projection a riemannian submersion.

Consider the Hopf fibration $S^1\ra S^3\ra S^2$. Let $\gamma$ be a
simple closed curve on $S^2$. Pick a point in $S^3$ over
$\gamma(0)$, and take the unique horizontal lift $\wt\gamma$ of
$\gamma$. Since $\gamma(1)=\gamma(0)$, $\wt\gamma(1)$ lies in the
same fiber as $\wt\gamma(0)$ does. We are interested in
understanding the difference between $\wt\gamma(0)$ and
$\wt\gamma(1)$. The following equality was already known
\cite{Pin}:
 $$
V(\gamma)=e^{\frac{1}{2}  A(\gamma) i},
$$
where $V(\gamma)$ is  the holonomy
displacement along $\gamma$, and
$A(\gamma)$ is the area of the region surrounded by $\gamma$.

In this paper, we  generalize this fact to the following higher dimensional Stiefel bundle over the Grassmannian manifold through $U_{m,n}(\bbc)$
$$
   U(n) \ra U(n+m)/U(m) \stackrel{\pi}\ra G_{n,m},
$$
where $G_{n,m}=U(n+m)/\big(U(n) \times U(m)\big).$
\noindent
The main results are stated as follows:
For $\hat{X} \in {\mathfrak u}(n+m)$ which is induced by
$X \in U_{m,n}(\bbc)$, 
consider a 2-dimensional subspace 
$\frakm' \subset \frakm  $ $ \subset \mathfrak{u}(m+n) $ 
with $\hat{X} \in \frakm'.$
Assume either 
$$\frakm' = \text{Span}_{\bbr}\{\hat{X},\hat{Y}\}$$
for some $Y \in U_{m,n}$ with $X^{*}Y = \mu I_n$ for some $\mu \in \bbr$
or
$$\frakm' = \text{Span}_{\bbr} \{\hat{X},\widehat{iX}\}.$$
Then $\frakm'$  gives rise to a complete totally geodesic surface $S$ in the base
space. Furthermore,
let $\gamma$
be a piecewise smooth, simple closed curve on $S$
parametrized by $0\leq t\leq 1$, and $\wt\gamma$ its horizontal lift on
the bundle
$U(n) \ra \pi^{-1}(S) \stackrel{\pi}{\rightarrow} S,$
which is immersed in
$U(n) \ra U(n+m)/U(m) \stackrel{\pi}\ra G_{n,m} $.
Then
$$
\wt\gamma(1)= \wt\gamma(0) \cdot ( e^{i \theta} I_n)
\text{\hskip24pt or\hskip12pt } \wt\gamma(1)= \wt\gamma(0),
$$
depending on whether the immersed bundle is flat or not,
where $A(\gamma)$ is the area of the region on the surface $S$
surrounded by $\gamma$ and 
$\theta= 2 \cdot \tfrac{n+m}{2n} A(\gamma).$
See {\rm Theorem  \ref{thm-sphere}}.

\section{
            The bundle
             $
                S^1 \ra SU(2) \ra \bbc P^1  
             $ 
           } \label{hopf}

It will be studied not only the holonomy displacement of the bundle 
$S^1 \ra SU(2) \ra \bbc P^1$ 
but also its isomorphic equivalence to the one
$$
  S\big( U(1) \times U(1) \big)
  \ra
  SU(1+1)
  \ra
  SU(1+1) / S\big(U(1) \times U(1)),
$$ 
not the isometric equivalence.
In fact, a conformal map $h: SU(1+1) / S\big(U(1) \times U(1)) \ra \bbc P^1$ will be constructed such that the identity map on $SU(2)$ is the bundle map covering  it.
The latter bundle will play an important role for the case
$\frakm' = \text{Span}_{\bbr} \{\hat{X},\widehat{iX}\}.$

Of course,
$$
S^3\cong SU(2)=\{A\in\GL(2,\bbc)\ :\  A^* A=I\text{ and } \det(A)=1\}
$$
for $S^3 = \{(z_1, z_2) \big| |z_1|^2 + |z_2|^2 =1\}$
under the map
$$
  (z_1, z_2) 
  \mapsto 
  \left[\begin{array}{cccc}
        \bar{z}_1 & - \bar{z}_2\\
        z_2 & z_1\\
        \end{array}
  \right]
  : S^3 \ra SU(2).
$$

From now on, we  use the convention of
$\frak{gl}(k,\bbc)\subset\frak{gl}(2k,\bbr)$ by
$$
\left[\begin{array}{cccc}
z_{11} & z_{12}\\
z_{21} & z_{22}\\
\end{array}\right]
\lra
\left[\begin{array}{cccc}
x_{11} +i y_{11}  & x_{12} +i y_{12}  \\
x_{21} +i y_{21}  & x_{22} +i y_{22} \\
\end{array}\right]
\lra
\left[\begin{array}{rrrrrrrr}
x_{11} &-y_{11} &x_{12} &-y_{12}\\
y_{11} & x_{11} &y_{12} & x_{12}\\
x_{21} &-y_{21} &x_{22} &-y_{22}\\
y_{21} & x_{21} &y_{22} & x_{22}\\
\end{array}\right],
$$
which is an isometric monomorphism with respect to the metric 
on $GL(k, \bbc)$ and on $GL(2k, \bbr),$ given by
$$
  \langle A,B \rangle 
  = \tfrac{1}{k}\text{Re}\big(\text{Tr}(A^{*}B) \big), 
      \qquad A,B \in \frak{gl}(k,\bbc)
$$
and
$$
  \langle C,D \rangle 
  = \tfrac{1}{2k}\text{Tr}(C^{t}D), \qquad C,D \in \frak{gl}(2k,\bbr),
$$
respectively.

\bigskip
The group $SU(2)$ has the following natural representation into
$\GL(4,\bbr)$:
$$
w=
\left[
\begin{array}{rrrr}
\fm w_1 &w_2 &-w_3 &-w_4\\
-w_2 &w_1 &w_4 &-w_3\\
w_3 &-w_4 &w_1 &-w_2\\
w_4 &w_3 &w_2 &w_1
\end{array}
\right]
$$
with the condition $w_1^2 + w_2^2 + w_3^2 + w_4^2 =1$.
In fact, the map
$$
w_1 + w_2 i + w_3 j + w_4 k
\longmapsto
w
$$
is  an isometric monomorphism  from the unit quaternions into $\gl(4,\bbr)$.
The circle group
$$
S\big( U(1) \times U(1) \big) =
\left\{
\left[
\begin{array}{ll}
e^{-iz}  &0  \\
0& e^{iz}  \\
\end{array}\right]\ :\ 0\leq z\leq 2\pi
\right\}
$$
\noindent
is a subgroup of $SU(2)$, and acts on $SU(2)$ as right translations,
freely with quotient 
$SU(1+1) / S\big(U(1) \times U(1)),$
which is an affine symmetric space and produces a principal circle bundle
$$
  S\big( U(1) \times U(1) \big)
  \ra
  SU(1+1)
  \ra
  G_{1,1}= SU(1+1) / S\big(U(1) \times U(1)).
$$

\bigskip
Let $\wt w$ be the ``$i$-conjugate'' of $w$ (replace $w_2$ by $-w_2$). That is,
$$
\wt w=
\left[
\begin{array}{rrrr}
\fm w_1 &-w_2 &-w_3 &-w_4\\
w_2 &w_1 &w_4 &-w_3\\
w_3 &-w_4 &w_1 &w_2\\
w_4 &w_3 &-w_2 &w_1
\end{array}
\right].
$$
Then,
$$
w \wt w=
{\tiny
\left[
\begin{array}{rrrr}
w_1^2+w_2^2-w_3^2-w_4^2 &0 &-2 (w_1 w_3+w_2 w_4) &2 w_2 w_3-2 w_1 w_4\\
0 &w_1^2+w_2^2-w_3^2-w_4^2 &-2 w_2 w_3+2 w_1 w_4 & -2 (w_1 w_3+w_2 w_4)\\
2 (w_1 w_3+w_2 w_4) &2 w_2 w_3-2 w_1 w_4 &w_1^2+w_2^2-w_3^2-w_4^2 &0\\
-2 w_2 w_3+2 w_1 w_4 &2 (w_1 w_3+w_2 w_4) &0 &w_1^2+w_2^2-w_3^2-w_4^2
\end{array}
\right]
}
$$
and
$$
(w_1^2+w_2^2-w_3^2-w_4^2)^2+(2 w_1 w_3+2 w_2 w_4)^2 +(-2 w_2 w_3+2 w_1
w_4)^2=1.
$$
Clearly, $\bbc P^1$ can be identified with the following
 $$
\bbc P^1 =
\left\{
\left[
\begin{array}{rrrr}
x & 0 &-y &-z  \\
0 & x & z &-y\\
y &-z & x & 0\\
z & y & 0 & x
\end{array}\right]
:\  x^2+y^2+z^2=1
\right\},
$$
which is a subset of $SU(2)$ such that $i$-conjugate on $\bbc P^1$ is the identity map of $\bbc P^1$.
And the map
$$
p: SU(2)\lra \bbc P^1
$$
defined by
$$
p(w)=w \wt w
$$
has the following properties:
\begin{align*}
p(wv)&=w p(v) \wt w\quad\text{for all } w,v\in SU(2)\\
p(wv)&=p(w) \quad\text{if and only if}\quad 
              v\in S\big(U(1) \times U(1)) \cong S^1
\end{align*}
under the convention of 
$S\big(U(1) \times U(1)) \hookrightarrow GL(4,\bbr).$
This shows that the map $p$ is, indeed, the orbit map of the
principal bundle 
$$S^1\lra SU(2) \stackrel{p}\lra \bbc P^1.$$
But we have to be careful that the inclusion map $\bbc P^1 \hookrightarrow SU(2)$ is not a cross-section in this bundle. In fact,
$p(v) = v^2 \in \bbc P^1$ for any $ v \in \bbc P^1.$

\bigskip
Define a map $h: SU(2)/S\big(U(1) \times U(1)\big) \ra \bbc P^1$ by
$$h(v H) = v^2 = p(v) \qquad v \in \bbc P^1,$$
where $H = S\big(U(1) \times U(1)\big).$
Then, the identity map of $SU(2)$ is a trivially isomorphic bundle map which covers the map $h.$
Under the identification 
$
 (x,y,z) =
 \left[
 \begin{array}{rrrr}
 x & 0 &-y &-z  \\
 0 & x & z &-y\\
 y &-z & x & 0\\
 z & y & 0 & x
 \end{array}\right]
 : S^2 \cong \bbc P^1,
$
give the metric $\langle \cdot , \cdot \rangle _{S^2}$ of $S^2$ to $\bbc P^1$ and consider a metric space 
$\big( \bbc P^1, \langle \cdot , \cdot \rangle _{S^2} \big) .$ Will $h$ be an isometry?

\bigskip
The Lie group $SU(2)$ will have a left-invariant Riemannian metric
given by the
following orthonormal basis on the Lie algebra $\mathfrak{su}(2) $
$$
  E_1 =
  \left[
          \begin{array}{cc}
            0 & -1 \\ 1 & 0
          \end{array}
  \right]\ ,
  \quad
  E_2 =
  \left[
          \begin{array}{cc}
            0 & i \\ i & 0
          \end{array}
  \right]\ ,
  \quad
  E_3 =
  \left[
          \begin{array}{cc}
            -i & 0 \\ 0 & i
          \end{array}
  \right]\ ,
$$
which correspond to
\[
e_{1}= \left( \begin{array}{rrrr}
0 & 0 & -1& 0  \\
0 & 0 & 0  & -1\\
1 & 0 & 0 &0\\
0 & 1 & 0 &0
\end{array}\right),\;\;
e_{2}= \left( \begin{array}{rrrr}
0 & 0 & 0& -1  \\
0 & 0 & 1  & 0\\
0 & -1 & 0 &0\\
1 & 0 & 0 &0
\end{array}\right) ,\;\;
e_{3}=
\left( \begin{array}{rrrr}
0 & 1 & 0& 0  \\
-1 & 0 & 0  & 0\\
0 & 0 & 0 &-1\\
0 & 0 & 1 &0\\
\end{array}\right) 
\]
in $\frak{gl}(2k,\bbr),$ respectively.
Notice that 
$[e_1,e_2]=2 e_3$.

\bigskip
In order to understand 
the map $h$ between base spaces and the projection map $p$ better, consider the subset of $SU(2)$:
\begin{align*}
T&=
\left\{
\left[\begin{array}{ll}
\cos x &-(\sin x)e^{-i y}\\
(\sin x)e^{i y} &\cos x\\
\end{array}\right]\ :\ 0\leq x\leq \pi,\ 0\leq y\leq 2\pi
\right\}\\
&=
\left\{
\left[\begin{array}{rrrrrrrr}
\cos x &0          &-(\sin x)(\cos y) &-(\sin x)(\sin y)\\
0          &\cos x & (\sin x)(\sin y) &-(\sin x)(\cos y)\\
 (\sin x)(\cos y) &-(\sin x)(\sin y) &\cos x &0         \\
 (\sin x)(\sin y) & (\sin x)(\cos y) &0          &\cos x\\
\end{array}\right]
\right\} \\
\end{align*}
which is the exponential image of
$$
\frakm=
\left\{\left[\begin{array}{cc}
 0 &  -\bar{\xi}^{t} \\ \xi &0
\end{array} \right]\ :\ \xi  \in  {\bbc} \right\}.
$$
Furthermore, it is exactly same as $\bbc P^1,$ so the map $p$ restricted to $T$ is just the squaring map; that is,
$$
p(w)=w^2,\quad w\in T.
$$

\medskip
To check $h$ is a conformal map:
given 
$$
  w =\big(\cos{x}, (\sin{x})(\cos{y}), (\sin{x})(\sin{y}) \big) \in T = \bbc P^1,$$
\begin{align*}
  \big| D_1(wH) \big| 
  &= \big|(D_1 w)^{\text{h}} \big| \\
  &= \big|
           \big(
                 (\cos{y}) {L_w}_{*} e_1 + (\sin{y}) {L_w}_{*} e_2
           \big)^{\text{h}} 
     \big| \\
  &= 1
\end{align*}
and
\begin{align*}
  &\big| D_2(wH) \big| = \big|(D_2 w)^{\text{h}} \big| \\
  &= \big| 
           \big( -\tfrac{1}{2} (\sin{2x})(\sin{y}) {L_w}_{*} e_1 
                   + \tfrac{1}{2} (\sin{2x})(\cos{y}) {L_w}_{*} e_2
                   - (\sin^2 {x}) {L_w}_{*} e_3
           \big)^{\text{h}} 
     \big| \\
  &= \tfrac{1}{2} \big| \sin{2x} \big|,
\end{align*}
while, under the expression $\langle a,b,c \rangle$ of vectors in $\bbr ^3,$
\begin{align*}
  \big| D_1 \, h(wH) \big| 
  &= \big| D_1 \, w^2 \big| \\
  &= \big|
         \langle 
                 -2\sin{2x}, \, 2 (\cos{2x}) (\cos{y}), \, 2 (\cos{2x}) (\sin{y})
         \rangle 
     \big| \\
  &= 2
\end{align*}
and
\begin{align*}
  \big| D_2 \, h(wH) \big| 
  &= \big| D_2 \, w^2 \big| \\
  &= \big|
         \langle 
                 0, \, - (\sin{2x}) (\sin{y}), \,  (\sin{2x}) (\cos{y})
         \rangle 
     \big| \\
  &= \big| \sin{2x} \big|.
\end{align*}
Thus $h$ is a conformal map.

\bigskip

\begin{theorem}[\cite{Pin}]
\label{area-u2}
Let $S^1\ra SU(2)\ra \Big( \bbc P^1, \langle \cdot , \cdot \rangle _{S^2} \Big)$ be the natural fibration.
Let $\gamma$ be a piecewise smooth, simple closed curve on $\bbc P^1$.
Then the holonomy displacement along $\gamma$ is given by
$$
  V(\gamma)= e^{\tfrac{1}{2} A(\gamma) i}  
           = e^{2 \cdot A( h^{-1} \circ \gamma) \Phi} 
             \in S^1 \cong S\big( U(1) \times U(1) \big)
$$
where $A(\gamma)$ is the area of the region on $\bbc P^1$
enclosed by $\gamma$ and 
$$
 \Phi = 
  \left[
         \begin{array}{cc}
             i &  0 \\ 0 & -i
         \end{array} 
  \right].
$$
\end{theorem}

\begin{proof}
Let $\gamma(t)$ be a closed loop on $\bbc P^1$ with $\gamma(0)=p(I_4)$.
Therefore,
{\tiny
\begin{align*}
\gamma(t)=
\left[\begin{array}{rrrrrrrr}
\cos 2x(t) &0          &-\sin 2x(t)\cos y(t) &-\sin 2x(t)\sin y(t)\\
0          &\cos 2x(t) & \sin 2x(t)\sin y(t) &-\sin 2x(t)\cos y(t)\\
 \sin 2x(t)\cos y(t) &-\sin 2x(t)\sin y(t) &\cos 2x(t) &0         \\
 \sin 2x(t)\sin y(t) & \sin 2x(t)\cos y(t) &0          &\cos 2x(t)\\
\end{array}\right]
\end{align*}}
Let
\begin{align*}
{\tiny
\wt\gamma(t)=
\left[\begin{array}{rrrrrrrr}
\cos x(t) &0          &-\sin x(t)\cos y(t) &-\sin x(t)\sin y(t)\\
0          &\cos x(t) & \sin x(t)\sin y(t) &-\sin x(t)\cos y(t)\\
 \sin x(t)\cos y(t) &-\sin x(t)\sin y(t) &\cos x(t) &0         \\
 \sin x(t)\sin y(t) & \sin x(t)\cos y(t) &0          &\cos x(t)\\
\end{array}\right]
}
\end{align*}
with $0\leq x(t)\leq\pi/2$
so that $p(\wt\gamma(t))=\gamma(t)$ ($\wt\gamma$ is a lift of $\gamma$),
and let
\[
\omega(t)=
\left[
\begin{array}{rrrr}
\cos z(t)  &  -\sin z(t) & 0& 0  \\
\sin z(t)  & \cos z(t)   & 0 & 0\\
 0& 0 & \cos z(t) & \sin z(t) \\
 0& 0 & -\sin z(t) & \cos z(t)  \\
\end{array}\right].
\]
Put
$$
\eta(t)=\wt\gamma(t)\cdot\omega(t).
$$
Then still $p(\eta(t))=\gamma(t)$, and $\eta$ is another lift of
$\gamma$. We wish $\eta$ to be the horizontal lift of $\gamma$.
That is, we want $\eta'(t)$ to be orthogonal to the fiber at $\eta(t)$.

The condition is that
$\angles{\eta'(t),(\ell_{\eta(t)})_{*}(e_{3})}=0$, or equivalently,
$\angles{(\ell_{\eta(t)\inv})_*\eta'(t),e_3}=0$.
That is,
$$
\eta(t)^{-1} \cdot \eta'(t)= \alpha_{1}e_1  + \alpha_{2}e_{2}
$$
for some $\alpha_1,\alpha_2 \in\bbr$.
From this, we get the following equation:
\begin{equation}
\label{shpere-z}
z'(t)=\sin^2 x(t) y'(t).
\end{equation}

Since any piecewise smooth curve can be approximated by a sequence of
piecewise linear curves which are sums of boundaries of rectangular
regions,
it will be enough to
prove the statement for a particular type of curves as follows ~\cite{CL}:
Suppose we are given a rectangular region in the $xy$-plane
\begin{align*}
p \leq x \leq p+a,\qquad
q \leq y \leq q+b.
\end{align*}
Consider the image $R$ of this rectangle in $\bbc P^1$ by the map
$$
(x,y)\mapsto \bfr(x,y)=(\cos 2x, (\sin 2x)(\cos y), (\sin 2x)(\sin y)).
$$

\noindent Then  $||\bfr_x\x\bfr_y||=
2 \sin 2x$, (because $0\leq x\leq \pi/2$). Thus, the area of $R$ is
\begin{align*}
\int_q^{q+b}\int_p^{p+a} 2\sin 2x\ dx dy=2 b(\sin^2(p+a)-\sin^2(p)).
\end{align*}

On the other hand, the change of $z(t)$ along the boundary $\gamma(t)$ of this region
can be calculated using  condition (\ref{shpere-z}). Let  $\gamma(t)$  be represented by
$(p+4at,q)$          for  $t\in [0, \frac{1}{4}]$,
$(p+a,q+b(4t-1)) $  for  $t\in [\frac{1}{4}, \frac{1}{2}]$, $(p+a(3-4t),q+b)$ for $ t\in [\frac{1}{2}, \frac{3}{4}]$,
$(p,q+b(4-4t)) $ for $t\in [\frac{3}{4}, 1]$. Then
$$
z(1)-z(0)
=\int_0^1 z'(t) dt
=b\cdot\sin^2 (p+a)- b\cdot\sin^2 (p).
$$

Thus the total vertical  change of $z$-values, $z(1)-z(0)$,  along the perimeter of this rectangle
is
$$
b\cdot(\sin^2 (p+a)-\sin^2 (p))
$$
which is $\frac{1}{2}$ times  the area. Hence we get the conclusion.
\end{proof}
\bigskip

\section{
             The bundle
             $U(n)\lra U(n+m)/U(m)\lra G_{n,m}$
           }

To deal with the bundle
\[ U(n) \rightarrow U(n+m)/U(m) {\rightarrow}  G_{n,m}\, , \]
we investigate the bundle
\[ U(n) \times U(m) \rightarrow U(n+m) {\rightarrow}  G_{n,m}\, . \]

The Lie algebra of $U(n+m)$ is $\fraku(n+m)$, the skew-Hermitian
matrices, and has the following canonical decomposition:
$$
\frakg=\frakh +\frakm,
$$
where
$$
\frakh=\fraku(n)+\fraku(m)=\left\{
\left[\begin{array}{cc}
A  & 0\\ 0 & B \end{array}\right]\ :\
A\in \fraku(n),\  B\in \fraku(m)\right\}
$$
and
$$
\frakm=
\left\{
  \hat{X} \!:=
  \left[
          \begin{array}{cc}
            0 & - X^* \\ X &0
          \end{array}
  \right]\
  :\ X  \in  M_{m,n}(\bbc)
\right\}.
$$

\vspace{0.3cm}
  Define an Hermitian inner product $h : \bbc ^m \rightarrow \bbc$ by
$$h(v,w) = v^*\, w,$$
where $v$ and  $w$ are regarded as column vectors.

\begin{lemma}
\label{lambda}
If a matrix $X \in M_{m,n}$ satisfies $X^{*}X = \lambda I_n$ for some $\lambda \in \bbc,$ then $\lambda$ will be a nonnegtive real number and $\lambda=0$ only if $X$ is trivial. 
\end{lemma}
\begin{proof}
Given any column vector $v$ of $X$,  $\lambda = v^{*}v = h(v,v) \ge 0$ and the equality holds only if $v=0,$ which shows the claim.
\end{proof}

\bigskip
From the lemma \ref{lambda}, we obtain that
\begin{align*}
  U_{m,n}(\bbc)
  &= 
  \{ 
     X \in M_{m,n}(\bbc) 
     \,\, | \,\, 
     X^{*} X = \lambda I_n \text{ for some } \lambda \in \bbc - \{0\}
  \}
  \\
  &=
  \{ 
     X \in M_{m,n}(\bbc) 
     \,\, | \,\, 
     X^{*} X = \lambda I_n \text{ for some } \lambda > 0
  \}.
\end{align*}

\bigskip
\begin{lemma}
\label{calculation}
   Let
   $$
      X = \Big(a^r_k + i b^r_k \Big) \, ,
      Y = \Big(c^r_k + i d^r_k \Big)
      \in M_{m,n}(\bbc)$$
for  $r=1, \cdots, m$, and  $ k=1, \cdots, n$. Suppose that  for their induced
 $\hat{X}, \, \hat{Y} \in \frakm$,
 $$
    [[\hat{X}, \, \hat{Y}], \hat{X}]  = \hat{Z} \in \frakm
 $$
 for some
  $
   Z  =
   \Big(\alpha ^r_k \Big) \in M_{m,n}(\bbc)
  $ for $r=1, \cdots, m$, and  $k=1, \cdots, n$.
  Then we have
 $$
   \alpha ^r_k =
   \sum^{n}_{j=1} (a^r_j + i b^r_j) \big(\!-2 h(Y_j, X_k) + h (X_j, Y_k)\big)
   + \sum^{n}_{j=1} (c^r_j + i d^r_j) \, h(X_j, X_k),
 $$
where $X_k$ and $Y_k$ are $k$-column vectors of $X$ and $Y$ for $k= 1, \cdots, n.$
\end{lemma}

\begin{proof}
It is easily obtained from
$$
  [[\hat{X},\hat{Y}], \hat{X}] =
  \hat{X}(2 \hat{Y}\hat{X}-\hat{X}\hat{Y})
   - \hat{Y}\hat{X}\hat{X}.$$
\end{proof}

  Recall the following proposition, which gives the clue for the holonomy displacement in the principal $U(n)$ bundles over Grassmaniann manifolds 
$U(n) \ra U(n+m)/U(m) \stackrel{\pi} \ra G_{n.m}$.
\begin{proposition} \cite{KN} \label{affine}
 Let $(G,H,\sigma)$ be a symmetric space and
 $\mathfrak{g} = \mathfrak{h} + \mathfrak{m}$
 the canonical decomposition. Then there is a natural one-to-one correspondence between the set of linear subspaces $\mathfrak{m}'$ of $\mathfrak{m}$ such that
$[[\mathfrak{m}', \mathfrak{m}'], \mathfrak{m}'] \subset \mathfrak{m}'$
and the set of complete totally geodesic submanifolds $M'$ through the origin $0$ of the affine symmetric space $M=G/H,$ the correspondence being given by 
$\mathfrak{m}' = T_0 (M').$
\end{proposition}

Note that $\mathfrak{m}'$ in the Proposition \ref{affine} will make a bunch of complete totally geodesic submanifolds, each of which is obtained from another one by a translation, in the affine symmetric space $G/H.$

\bigskip
The role of $U_{m,n}(C)$ in this paper will be seen from now on.

\bigskip

\begin{theorem} \label{easy}
Given $X \in U_{m,n}(\bbc)$ and
the natural fibration
$
 U(n) \times U(m) \ra U(n+m) \ra G_{n,m}(\bbc),
$
assume a 2-dimensional subspace $\frakm' = \text{Span}_{\bbr} \{ \hat{X}, \hat{Y}\} $ of
$\frakm\subset {\mathfrak u}(n+m)$
satisfies
\begin{align}
   X^* \, X = \lambda I_n, \quad  X^* \, Y = \mu I_n, \qquad 
    \mu \in \bbc \label{STAR}
\end{align}
for  $ Y \in M_{m,n}(\bbc).$
Then $\frakm'$ gives rise to a complete totally geodesic  surface $S$ in 
$G_{n,m}(\bbc)$ 
if and only if 
$
 \text{Im}\,{\mu} = 0
$  
and $Y \in U_{m,n}(\bbc)$ 
\end{theorem}

\begin{proof}
 To begin with, note that $\lambda >0.$
 Assume that  $\frakm'$ gives rise to a complete totally geodesic  surface $S$  in $G_{n,m}(\bbc)$.
By a translation, without loss of generality, we can assume that $S$ passes through the origin of the affine symmetric space 
$G_{n,m}(\bbc) = U(n+m)/\left(U(n) \times U(m)\right)$

 To show $\text{Im}\mu = 0$ by contradiction, suppose that  
$\text{Im}\mu \not= 0$. Let $e_k \in \bbc^m, \: k=1, \cdots, m,$ be an elementary vector which has all components 0 except for the $k$-component with 1. Then
  $$h(X_k, Y_j) = h(X e_k, Y e_j) = e_k ^* (X^* Y) e_j ,$$
  so the condition (\ref{STAR}) is equivalent to
  $$
     h(X_k, Y_k) = \mu, \quad
     h(X_k, X_k) = \lambda, \quad
     h(X_k, X_j) =0 , \quad
     h(X_k, Y_j) =0
  $$
  \noindent  
  for $k \not= j $ in $\{1, \cdots ,n \}$.  From $ h(X_k, Y_k) = \mu $,
 we obtain
  $$
    -2h(Y_k, X_k) + h(X_k, Y_k) =
    - \text{Re} \mu + 3i \text{Im}  \mu.
  $$
  \noindent
  Thus Lemma ~\ref{calculation}, Proposition \ref{affine} and the hypothesis of totally geodesic say that \\
  \begin{align*}
    a \hat{X} + b \hat{Y}
    =
    [[\hat{X}, \hat{Y}], \hat{X}]
    &= (-\text{Re}\mu + 3i \text{Im}\mu) \hat{X} + \lambda \hat{Y}
    \\
    &=  3 \text{Im}\mu (i\hat{X})
          + (-\text{Re}\mu  \hat{X} + \lambda \hat{Y}) .
  \end{align*}
for some $a, b \in \bbr$.
  Since $\text{Im}\mu \not= 0$, $i \hat{X}$ will lie in $\text{Span}_{\bbr} \{\hat{X}, \hat{Y} \} = \frakm' \subset \mathfrak{u}(n+m)$, and then
  $$
    -i\hat{X} = - (i \hat{X}) = (i\hat{X})^{*} = -i \hat{X}^{*} =  i \hat{X},
  $$
  \noindent
  which implies $\hat{X}=O_{n+m}$, a contradiction. 
 
  From $\text{Im}\mu = 0$,
  $$
    -X^* Y +  Y^* X = - X^* Y + (X^* Y)^*
    =  - 2 i \text{Im}\mu \, I_n = O_n,
  $$ 
  \noindent 
  so
  $$
  [\hat{X}, \hat{Y}] =
  \left[
          \begin{array}{cc}
            O_n & 0 \\ 0 & - X Y^* + Y X^*
          \end{array}
  \right]\
  \in \mathfrak{u} (m) \subset \mathfrak{u} (n+m).
  $$
  Let $M = - X Y^* + Y X^*$. Then
  $$
  [\hat{X}, \hat{Y}] =
  \left[
          \begin{array}{cc}
            O_n & 0 \\ 0 & M
          \end{array}
  \right]\
  $$
  and
  $ [[\hat{Y}, \hat{X}], \hat{Y}] = - \widehat{MY} \in \frakm'$
  from the hypothesis of the condition of totally geodesic and from
  Proposition \ref{affine}.
  Note that
  $$
    -MY =  XY^*Y - Y X^*Y = XY^*Y - Y \mu I_n
           = XY^*Y- (\text{Re}u) Y.
  $$
  Thus $XY^*Y = aX + bY$  for some $a, b \in \bbr $. Then
  $
    \lambda Y^*Y = X^*(XY^*Y)
    = X^*( aX + bY) = (a \lambda + b \text{Re}\mu) I_n
  $
  and so
  $$
    Y^*Y= \tfrac{ a \lambda + b \text{Re}\mu}{\lambda}I_n, \quad 
              \tfrac{ a \lambda + b \text{Re}\mu}{\lambda}  \in \bbr.
  $$
 Since $\mathfrak{m}'= \text{Span}_{\bbr} \{ \hat{X}, \hat{Y}\}$ is 2-dimensional, $Y$ is not a zero matrix and so from Lemma \ref{lambda}, 
 $Y \in U_{m,n}(\bbc)$.

Conversely, assume the necessary part holds and let $Y^{*}Y = \eta I_n$, where $\eta > 0$.
Then, the condition $\text{Im}\mu = 0$ says that 
  $$
  [\hat{X}, \hat{Y}] =
  \left[
          \begin{array}{cc}
            O_n & 0 \\ 0 & M
          \end{array}
  \right]\ , \,\,
  [[\hat{X}, \hat{Y}], \hat{X}] = \widehat{MX} \,\,
  \text{ and } \,\,
  [[\hat{Y}, \hat{X}], \hat{Y}] = -\widehat{MY},
  $$  
  where $M = - X Y^* + Y X^*$. It suffices to show that $[[\hat{X}, \hat{Y}], \hat{X}] \in \frakm'$  and
   $[[\hat{Y}, \hat{X}], \hat{Y}] \in \frakm'$.
  Since
  $$
    MX = - XY^*X + Y X^*X = - X \bar{\mu} I_n + Y \lambda I_n
          = - \text{Re}\mu X + \lambda Y,
  $$ \noindent  we get  $[[\hat{X}, \hat{Y}], \hat{X}] \in \frakm'$.
  We also get $[[\hat{Y}, \hat{X}], \hat{Y}] \in \frakm'$ since
  $$
    -MY =  XY^*Y - Y X^*Y
           = X \eta I_n - Y \mu I_n
           = \eta X - \text{Re}\mu Y.
  $$
  Hence we get the conclusion.
\end{proof}

\begin{corollary}
Given $X,Y \in U_{m,n}(\bbc)$ and given the natural fibration
$
 U(n) \times U(m) \ra U(n+m) \ra G_{n,m}(\bbc),
$
assume $\frakm' = \text{Span}_{\bbr} \{ \hat{X}, \hat{Y}\} $ produce a 2-dimensional subspace of $\frakm\subset {\mathfrak u}(n+m).$ 
If $X^* \, Y = \mu I_n \text{ for some } \mu \in \bbr,$
then 
$\frakm'$ will give rise to a complete totally geodesic  surface $S$ in 
$G_{n,m}(\bbc)$ 
\end{corollary}

\begin{remark}
Given $X \in U_{m,n}(\bbc), $ if $n \le m, $ then 
$X: \bbc^n \ra \bbc^m$ is a conformal one-one linear map. In view of
$\hat{X} \in {\mathfrak u}(n+m) \subset \text{End}(\bbc^{n+m})$,
$\hat{X}$ sends the subspace $\bbc^n$ to its orthogonal subspace $\bbc^m$ conformally. And the condition of the relation between $X$ and $Y$ in Theorem ~\ref{easy} says that
$$
   h_{\bbc^m}(Xv, Yw) = \mu \: h_{\bbc^n}(v,w) \qquad
   \text{ for } v, w \in \bbc^n,
$$
where $h_{\bbc^k}$ is an Hermitian on
$\bbc^k, \: k= 1,2, \cdots$, given by
$$
   h_{\bbc^k}(u_1,u_2) = u_1^* u_2 \quad \text{ for } u_1, u_2 \in \bbc^k.$$
\end{remark}
\bigskip

 When $n=1$, the condition (\ref{STAR}) is satisfied automatically for any two vectors in $\bbc^m$ by identifying
$M_{m,1}(\bbc)$ with $\bbc^m$.
So we get

\begin{corollary}
\label{geod-cond-cpn}
A 2-dimensional subspace $\frakm'$ of
$\frakm\subset {\mathfrak u}(m+1)$
gives rise to a complete totally geodesic submanifold in the affine symmetric space
${\bbc}P^m = U(1+m)/ \left(U(1) \times U(m) \right)$
if  
$\frakm'$ has two linearly independent tangent vectors $\hat{v}$ and $\hat{w}$ such that
$\text{Im}h_{\bbc^m}(v,w) =0$.
\end{corollary}

\bigskip
We return to the bundle
$U(n) \ra U(n+m)/U(m) \stackrel{\pi}\lra G_{n,m}$.
Any submanifold $A \subset G_{n,m}$ induces a bundle
$U(n) \ra \pi^{-1}(A) \ra A$,
which is immersed in the original bundle and diffeomorphic to the pullback bundle with respect the inclusion of $A$ into $G_{n,m}$. In fact, in the bundle
$U(n)\times U(m) \ra U(n+m) \stackrel{\tilde{\pi}}\lra G_{n,m}$, the induced distribution in $\tilde{\pi}^{-1}(A)$
from ${\mathfrak u}$(m) in $U(n+m)$ is integrable and preserved by the right multiplication of $U(n),$ so this induces the bundle $U(n) \ra \pi^{-1}(A) \ra A$.

\begin{theorem}
\label{geod-cond-sphere1}
Given a complete totally geodesic  surface $S$ in $G_{m,n}$ which is induced by a 2-dimensional subspace $\mathfrak{m}' \subset \mathfrak{m}$ with the necessary condition in Theorem ~\ref{easy} satisfied, 
the bundle
$U(n) \ra \pi^{-1}(S) \ra S$,
which is immersed in the original bundle
$U(n) \ra U(n+m)/U(m) \stackrel{\pi}\lra G_{n,m}$,
is flat.
\end{theorem}

\begin{proof}
By a left translation, without loss of generality, assume that $S$ passes through the origin of the affine symmetric space $G_{n,m}.$

Consider the bundle
$U(n)\times U(m) \ra U(n+m) \stackrel{\tilde{\pi}}\lra G_{n,m}.$
Then $S$ induces a bundle
$U(n)\times U(m) \ra \tilde{\pi}^{-1}(S) \ra S$.
Totally geodesic condition says that the distribution induced from
$\text{Span}_{\bbr} \{ \hat{X}, \hat{Y}, [\hat{X},\hat{Y}] \}$ is integrable.
Since $[\hat{X},\hat{Y}]$  is contained in the Lie algebra   $ {\mathfrak u}(m) $ of $U(m)$ from the proof of Theorem \ref{easy},  the conclusion is obtained.
\end{proof}

\bigskip

\begin{theorem}
\label{easy_generalizion}
\label{geod-cond-sphere}
Given $X \in U_{m,n}(\bbc)$ and the natural fibration
$
 U(n) \times U(m) \ra U(n+m) \stackrel{\tilde{\pi}}\lra G_{n,m}(\bbc),
$
consider the 2-dimensional subspace $\frakm' = \text{Span}_{\bbr} \{ \hat{X}, \widehat{iX}\} $. Then,
\begin{enumerate}
\item
$\frakm'$ gives rise to a complete totally geodesic  surface $S$ in 
$G_{n,m}(\bbc),$
\item
$\frakm'$ induces a $U(1)$-subbundle of a bundle 
$$U(n) \times U(m) \ra \tilde{\pi}^{-1}(S) \ra S,$$
which is an immersion of the bundle
$$
  S\big( U(1) \times U(1) \big)
  \ra
  SU(1+1)
  \ra
  SU(1+1) / S\big(U(1) \times U(1))
$$
into 
$$U(n) \times U(m) \ra U(n+m) \stackrel{\tilde{\pi}}\lra G_{n,m},$$
such that it is isomorphic to
the Hopf bundle $S^1 \ra S^3 \ra S^2,$
\item
the immersion is conformal, and isometric in case of $n=m.$ In fact, 
$$ \big| \tilde{f}_{*}v \big| = \sqrt{\tfrac{2n}{n+m}} \, |v|$$ 
under the expression $\tilde{f}: SU(2) \ra U(n+m)$ for the immersion.
\end{enumerate}
\end{theorem}

\begin{proof}
From Lemma \ref{lambda}, let $X^* \, X = \lambda I_n$ for some $\lambda>0$.

By a left translation, without loss of generality, assume that $S$ passes through the origin of the affine symmetric space $G_{n,m}.$

Note that,
for
$
  K =
  \left[
          \begin{array}{cc}
            - i \lambda I_n & 0 \\ 0 & iXX^{*}
          \end{array}
  \right]\
  \!\in {\mathfrak {u}}(n) \times {\mathfrak {u}}(m),
$
$$
  [\hat{X},\widehat{iX}]=2 K,  \quad
  [K,\hat{X}]=2 \lambda \widehat{iX},  \quad
  [K,\widehat{iX}]=-2 \lambda \hat{X},
$$
which implies $[[\frakm',\frakm'],\frakm'] \subset \frakm'$ and the conclusion (1). 

Consider an orthonormal basis of
 ${\mathfrak {su}}(1+1)$:
$$
  E_1 =
  \left[
          \begin{array}{cc}
            0 & -1 \\ 1 & 0
          \end{array}
  \right]\ ,
  \quad
  E_2 =
  \left[
          \begin{array}{cc}
            0 & i \\ i & 0
          \end{array}
  \right]\ ,
  \quad
  E_3 =
  \left[
          \begin{array}{cc}
            -i & 0 \\ 0 & i
          \end{array}
  \right]\ ,
$$
and a Lie algbra monomorphism
$f : {\mathfrak su}(1+1) \ra {\mathfrak u}(n+m)$, given by
$$
  f(aE_1+bE_2+cE_3)
  = \frac{a}{\sqrt{\lambda}} \hat{X} 
    + \frac{b}{\sqrt{\lambda}}  \widehat{iX} 
    + \frac{c}{\lambda}  K  
$$
\noindent
for $a,b,c \in \bbr,$ from
$$[E_1,E_2]=2 E_3, \quad [E_3,E_1]=2E_2, \quad [E_3,E_2]=-2E_1.$$
For any $\theta \in \bbr,$ 
$$
  e^{\theta E_3} = 
  \left[
          \begin{array}{cc}
            e^{-i \theta} & 0 \\ 0 & e^{i \theta}
          \end{array}
  \right]\
  \in S \big( U(1) \times U(1) \big).  
$$ 
Thus $f$ will induce a Lie group monomorphism
$\tilde{f} : SU(1+1) \ra U(n+m)$ with
$
  \tilde{f}\Big({S \big( U(1) \times U(1) \big)}\Big)
  \subset
  U(n) \times U(m)
$
since $SU(2)$ is simply connected and
$S\big(U(1) \times U(1)\big)$ is connected.
Furthermore, it is the bundle map from
$$
  S\big( U(1) \times U(1) \big)
  \ra
  SU(1+1)
  \ra
  G_{1,1}= SU(1+1) / S\big(U(1) \times U(1))
$$
to
$$U(n) \times U(m) \ra U(n+m) \stackrel{\tilde{\pi}}\ra G_{n,m},$$
so the connected component of the integral manifold of the distribution induced by 
$\text{Span}_{\bbr}\{K,\hat{X}, \widehat{iX}\},$
which is the image of $\tilde{f},$
shows (2).

Note that 
$
 \{
   \frac{1}{\sqrt{\lambda}} \hat{X} , 
   \frac{1}{\sqrt{\lambda}}  \widehat{iX} ,
   \frac{1}{\lambda}  K  
 \}
$
is an orthogonal basis of the image of $\tilde{f}$ such that
$$
  \sqrt{\tfrac{2n}{n+m}} 
  = \Big| \tfrac{1}{\sqrt{\lambda}} \hat{X} \Big| 
  = \Big| \tfrac{1}{\sqrt{\lambda}}  \widehat{iX} \Big|
  = \Big| \tfrac{1}{\lambda}  K  \Big|,
$$
which shows (3).
\end{proof}

\bigskip
\begin{remark} \label{fiber}
 Let $\hat{\theta} = \tfrac{\theta}{\lambda}.$
 Then, for $\Phi = -E_3,$
  $$
   \tilde{f}(e^{\theta \Phi})
   = \tilde{f}(e^{-\theta E_3}) 
   = e^{-\hat{\theta} K}
   =
   \left[
          \begin{array}{cc}
            e^{i \theta} I_n & 0 \\ 
            0 & I_m + \tfrac{e^{-i \theta} -1}{\lambda} XX^{*}
          \end{array}
   \right]\ 
 $$
 from
 $$
    (-i \hat{\theta} XX^{*})^{j}
    = \Big(\tfrac{-i\theta}{\lambda}\Big)^{j} X(X^{*}X)^{j-1}X^{*}
    = \tfrac{(-i \theta)^{j}}{\lambda} XX^{*}
 $$
 for $j=1,2, \cdots .$
 Furthermore,
 \begin{align*}
   &\Big(I_m + \tfrac{e^{-i \theta} -1}{\lambda} XX^{*}\Big)
    \Big(I_m + \tfrac{e^{-i \phi} -1}{\lambda} XX^{*}\Big)  \\
   &=
    I_m + \tfrac{e^{-i \theta} + e^{-i \phi} -2}{\lambda} XX^{*}
    + \tfrac{e^{-i(\theta +\phi)}-e^{-i \theta}-e^{-i \phi} +1}{\lambda ^2}
       X(X^{*}X)X^{*} \\
   &= 
    I_m + \tfrac{e^{-i(\theta +\phi)} -1}{\lambda} XX^{*},    
 \end{align*}
 from which it is also obtained that
 $$
    I_m =
    \Big(I_m + \tfrac{e^{-i \theta} -1}{\lambda} XX^{*}\Big)
    \Big(I_m + \tfrac{e^{-i \theta} -1}{\lambda} XX^{*}\Big)^{*} .
 $$
\end{remark}

\bigskip
We return to the bundle
$U(n) \ra U(n+m)/U(m) \stackrel{\pi}\lra G_{n,m}$.
In fact, Remark \ref{fiber} implies that the immersed $U(1)$-subbundle, which is the image of $\tilde{f},$ gives two $U(1)$-bundles, one of which is an immersed $U(1)$-subbundle in the bundle
$U(n) \ra U(n+m)/U(m) \stackrel{\pi}\lra G_{n,m}$
and the other one is an immersed $U(1)$-subbundle in the bundle
$U(m) \ra U(n+m)/U(n) \stackrel{\hat{\pi}}\lra G_{n,m}.$

\bigskip
\begin{theorem}
\label{thm-sphere}
Assume the same condition for  a complete totally geodesic  surface
$S$  of either Theorem \ref{easy} or Theorem ~\ref{easy_generalizion},
and consider the immersed bundle
$U(n) \ra \pi^{-1}(S) \stackrel{\pi}{\rightarrow} S$
in the bundle
$U(n) \ra U(n+m)/U(m) \stackrel{\pi} \ra G_{n,m}.$ 
Let $\gamma$ be a piecewise smooth, simple closed curve on $S$.
Then the holonomy displacement along $\gamma$,
$$\wt\gamma(1)= \wt\gamma(0) \cdot V(\gamma),$$
is given by the right action of
$$
  V(\gamma)=e^{i \theta} I_n \quad \text{or} \quad e^{0i} I_n \in U(n),
$$
depending on whether the immersed bundle is flat or not,
where $A(\gamma)$ is the area of the region on the surface $S$ surrounded
by $\gamma$ and $\theta= 2 \cdot \tfrac{n+m}{2n} A(\gamma) .$
Especially, $\theta = 2\cdot A(\gamma)$ in case of $n=m.$
\end{theorem}

\begin{proof}
If the immersed bundle is flat, then it is obvious that the holonomy displacement is trivial. \\
Assume the condition of Theorem ~\ref{easy_generalizion} for the immersed $U(1)$-subbundle, which is the image of $\tilde{f}.$
Consider the induced map $\hat{f}: B \ra S \subset G_{m,n}$ between base spaces from the bundle map $\tilde{f}: SU(2) \ra \text{Im}(\tilde{f}) \subset U(n+m),$ which is a monomorphism, where $B = SU(2)/S(U(1) \times U(1)).$
Let 
$\alpha = \sqrt{\tfrac{2n}{n+m}},$
$
 \theta 
 = 2\cdot \tfrac{n+m}{2n} A(\gamma) 
 = \tfrac{\alpha ^{-2}}{8} A(\gamma)
$ 
and 
$\hat{\theta} = \tfrac{\theta}{\lambda}.$ 
The Theorem ~\ref{easy_generalizion}, Theorem \ref{area-u2} and 
Remark \ref{fiber} say that
the holonomy displacement of $\gamma$ in the bundle 
$U(n) \times U(m) \ra \pi^{-1}(S) \stackrel{\pi}{\rightarrow} S,$
which is immersed in the bundle
$U(n) \times U(m) \ra U(n+m) \stackrel{\pi} \ra G_{n,m},$
is given by the right action of
\begin{align*}
  V(\gamma) 
  &= \tilde{f} \big( V(\hat{f}^{-1} \circ \gamma)  \big) \\
  &= \tilde{f} 
       \big( e^{2 \cdot A(\hat{f}^{-1} \circ \gamma) \Phi} \big)  \\
  &= \tilde{f} \big( e^{\theta \Phi} \big)  \\
  &=    \left[
          \begin{array}{cc}
            e^{i \theta} I_n & 0 \\ 
            0 & I_m + \tfrac{e^{-i \theta} -1}{\lambda} XX^{*}
          \end{array}
        \right]\ .
\end{align*}
Thus in the bundle $U(n) \ra \pi^{-1}(S) \stackrel{\pi}{\rightarrow} S,$ which is immersed in the bundle 
$U(n) \ra U(n+m)/U(m) \stackrel{\pi} \ra G_{n,m},$
the holonomy displacement is given by the right action of
$$
  V(\gamma) = e^{i \theta} I_n .
$$
\end{proof}

\bigskip

\begin{remark}
For $n=1$, we have the following Hopf bundle
 $S^1\ra S^{2m+1} \ra \bbc P^m, $ where $\bbc P^m$ is given by the quotient metric, so the projection is a Riemannian submersion.
Let $S$ be a complete totally geodesic
surface in $\bbc P^m$ 
and  $\gamma$ be a piecewise smooth, simple closed curve on $S$.
Identify $\bbc ^m \cong M_{m,1}(\bbc).$
If $S$ is induced by $\text{Span} \{ v, w \} \subset \bbc ^m $ with
$\text{Im}h_{\bbc ^m}(v,w)=0,$ then the holonomy displacement along $\gamma$ is trivial.
See Corollary \ref{geod-cond-cpn} and Theorem \ref{geod-cond-sphere1}.
If $S$ is induced by a two dimensional subspace with complex structure in
$\bbc ^m,$
then the holonomy displacement depends not only on the area of the region surrounded by $\gamma$ but also on $m$ unless $m=1.$ In case of $m=1,$ here, $\bbc P^m$ is isometric to $S^2 \Big( \tfrac{1}{2} \Big).$ Refer to the map $h$ defined in Section \ref{hopf}.
\end{remark}

\medskip
\begin{remark}
Let $U(m) \ra U(n+m)/U(n) \stackrel{\hat{\pi}} \ra G_{n,m}$ be the natural
fibration.
Assume the same condition for  a complete totally geodesic  surface
$S$  of  Theorem ~\ref{easy_generalizion},
and consider the bundle
$U(m) \ra \hat{\pi}^{-1}(S) \stackrel{\hat{\pi}}{\rightarrow} S$. Let $\gamma$ be a piecewise smooth, simple closed curve on $S$.
Then the holonomy displacement along $\gamma$ is given by the right action of
$$
V(\gamma)=I_m + \tfrac{e^{-i \theta} -1}{\lambda} XX^{*} \in U(m),
$$
which depends on $X,$ not only on $n$ and $m,$
where $\theta= 2 \cdot \tfrac{n+m}{2n} A(\gamma).$
\end{remark}

\bibliographystyle{amsplain}

\begin{thebibliography}{20}

\bibitem{KN}
 S. Kobayashi and K. Nomizu, Foundations of differential geometry, Vol. II, Reprint of the 1969 original. Wiley Classics Library, A Wiley-Interscience Publication, John Wiley and Sons, Inc., New York, 1996.

\bibitem{CL}
 Y. Choi,  K.B. Lee,  Holonomy displacements  in the  Hopf bundles over  $\bbc H^n$ and the
complex Heisenberg groups, J. Korean Math. Soc. 49, 733--743 (2012)





\bibitem{Pin}
 U. Pinkall, Hopf tori in $S^3$,
Invent. Math. 81, 379--386 (1985)




\end{thebibliography}

\end{document}